\documentclass[a4paper,12pt,reqno]{myarticle}

\usepackage{color}
\definecolor{red}{rgb}{1,0,0}
\definecolor{green}{rgb}{0,1,0}
\definecolor{blue}{rgb}{0,0,1}
\definecolor{refkey}{gray}{.625}
\definecolor{labelkey}{gray}{.625}

\let\oldmarginpar\marginpar
\renewcommand\marginpar[1]{\-\oldmarginpar[\raggedleft\footnotesize #1]%
{\raggedright\footnotesize #1}}

\usepackage{verbatim}

\usepackage[T1]{fontenc}
\usepackage{lmodern}
\usepackage[latin1]{inputenc}

\usepackage{setspace}
\usepackage{indentfirst}

\usepackage{geometry}
\usepackage{hyperref}

\usepackage[msc-links]{amsrefs}

\newcommand{\doi}[1]{\href{http://dx.doi.org/#1}{\texttt{#1}}}

\renewcommand\labelenumi{\upshape{\theenumi}\ }
\renewcommand{\theenumi}{(\alph{enumi})} 
\renewcommand\labelenumii{\theenumii .}
\renewcommand{\theenumii}{\arabic{enumii}} 
\renewcommand\labelenumiii{(\theenumiii)}
\renewcommand{\theenumiii}{\arabic{enumiii}} 
\renewcommand\labelenumiv{\theenumiv .}
\renewcommand{\theenumiv}{\Roman{enumiv}} 

\usepackage{amssymb}
\usepackage{mathrsfs}
\usepackage{stmaryrd}
\usepackage{mathtools}

\usepackage{amscd}

\usepackage{graphicx}
\usepackage[all]{xy}
\def\dar[#1]{\ar@<2pt>[#1]\ar@<-2pt>[#1]}

\usepackage{amsthm}

\theoremstyle{plain}
\newtheorem{prop}{Proposition}[section]
\newtheorem{lem}[prop]{Lemma}

\newtheorem{cor}[prop]{Corollary}
\newtheorem{thm}[prop]{Theorem}

\newtheorem*{prop*}{Proposition}
\newtheorem*{lem*}{Lemma}
\newtheorem*{sublem*}{Sublemma}
\newtheorem*{cor*}{Corollaire}
\newtheorem*{thm*}{Theorem}
\newtheorem*{hypo*}{Hypothesis}
\newtheorem*{question*}{Question}
\newtheorem*{conjecture*}{Conjecture}
\newtheorem*{scholum*}{Scholum}
\newtheorem{defn}[prop]{Definition}
\newtheorem*{defn*}{Definition}

\newtheoremstyle{slanted}
  {3pt}
  {3pt}
  {\slshape}
  {}
  {\bfseries}
  {.}
  {.5em}
  {}

\theoremstyle{slanted}
\newtheorem{ex}[prop]{Example}

\newtheorem*{ex*}{Example}
\newtheorem*{exs*}{Examples}
\newtheorem{rmk}[prop]{Remark}

\newtheorem*{rmk*}{Remark}
\newtheorem*{rmks*}{Remarks}

\newtheorem*{notation*}{Notation}

\theoremstyle{definition}

\newtheorem*{con*}{Construction}
\newtheorem*{note*}{Note}

\theoremstyle{remark}

\newtheorem*{warning*}{Warning}
\newtheorem*{shortnote*}{Note}
\newtheorem*{claim*}{Claim}
\newtheorem*{axiom*}{Axiom}

\DeclareMathOperator{\Ad}{Ad}

\DeclareMathOperator{\id}{id}

\DeclareMathOperator{\RealPart}{Re}

\DeclareMathOperator{\modular}{mod}
\DeclareMathOperator{\conjugate}{C}
\DeclareMathOperator{\rk}{rk}

\newcommand{\JJ}{\mathbb{J}} 

\newcommand{\CC}{\mathbb{C}}

\newcommand{\ZZ}{\mathbb{Z}}

\newcommand{\thalf}{\tfrac{1}{2}}
\newcommand{\half}{\frac{1}{2}}
\newcommand{\rond}{\circ}
\newcommand{\cc}[1]{\overline{#1}} 

\newcommand{\cinf}[1]{C^{\infty}(#1)}
\newcommand{\sections}[1]{\Gamma(#1)}
\newcommand{\XX}{\mathfrak{X}} 

\newcommand{\xto}[1]{\xrightarrow{#1}}

\newcommand{\isomorphism}{\cong}

\newcommand{\diese}{^{\sharp}}
\newcommand{\bemol}{^{\flat}}
\newcommand{\inv}{^{-1}}

\newcommand{\graded}{^{\scriptscriptstyle\bullet}}

\newcommand{\abs}[1]{\left\vert#1\right\vert}
\newcommand{\set}[1]{\left\{#1\right\}}
\newcommand{\st}{\;\text{s.t.}\;}

\newcommand{\clifford}[1]{\mathcal{C}(#1)}
\newcommand{\ba}[2]{[#1,#2]}
\newcommand{\bas}[2]{[#1,#2]_*}
\newcommand{\basH}[2]{[#1,#2]^H_*}

\newcommand{\ld}[1]{L_{#1}}
\newcommand{\ii}[1]{\iota_{#1}}
\newcommand{\duality}[2]{\langle #1 | #2\rangle}
\newcommand{\ip}[2]{(#1,#2)}

\newcommand{\lb}[2]{[#1,#2]}

\newcommand{\anchor}{a}
\newcommand{\anchors}{\anchor_*}
\newcommand{\dee}{d}
\newcommand{\dees}{\dee_*}
\newcommand{\deesH}{\dees^H}

\newcommand{\module}{\mathscr{L}}

\newcommand{\lad}{A}
\newcommand{\dal}{\lad^*}

\newcommand{\SqrtMinusOne}{i}
\newcommand{\LieDer}{L }

\newcommand{\Id}[1]{\id_{#1}}

\newcommand{\bX}{\boldsymbol{X}}
\newcommand{\bY}{\boldsymbol{Y}}
\newcommand{\bT}{\boldsymbol{T}}
\newcommand{\TX}{T\bX}
\newcommand{\TCX}{T_{\CC}\bX}
\newcommand{\TsCX}{T^*_{\CC}\bX}
\newcommand{\forms}[1]{\Omega^{#1}(\bX)}
\newcommand{\formss}[1]{\Omega^{#1}}
\newcommand{\tangent}[1]{T^{#1}\bX}
\newcommand{\cotangent}[1]{(T^{#1}\bX)^*}
\newcommand{\Toi}{\tangent{0,1}}
\newcommand{\Tio}{\tangent{1,0}}
\newcommand{\Tois}{\cotangent{0,1}}
\newcommand{\Tios}{\cotangent{1,0}}
\newcommand{\TYio}{T^{1,0}\bY}

\newcommand{\TYoi}{T^{0,1}\bY}
\newcommand{\NYio}{N^{1,0}\bY}
\newcommand{\KKK}{K}

\newcommand{\pd}{\partial}
\newcommand{\barpd}{\bar{\partial}}
\newcommand{\barpdH}{\bar{\partial}^H}
\newcommand{\TsX}{T^*\bX}
\newcommand{\EE}{\boldsymbol{E}}

\newcommand{\QAsH}{Q_{\dal_H}}

\newcommand{\QAsHHalf}{Q^{\thalf}_{\dal_H}}
\newcommand{\stdiso}{\tau}
\newcommand{\cliffact}{\cdot}

\newcommand{\lbsH}[2]{[#1,#2]_*^H}
\newcommand{\anchorsH}{\anchors^H}
\newcommand{\representation}{\nabla^H}
\newcommand{\bdee}{\breve{\dee}}

\newcommand{\bdeesH}{\bdee_*^H}
\newcommand{\DD}{\mathcal{D}}
\newcommand{\DDH}{\DD^H}
\newcommand{\altmap}{F}
\newcommand{\overlinethetasharp}{\overline{\theta}\diese}
\newcommand{\overlinethetaflat}{\overline{\theta}\bemol}
\newcommand{\overlineomegaflat}{\overline{\omega}\bemol}
\newcommand{\overlinepisharp}{\overline{\pi}\diese}
\newcommand{\Hsharp}{H\diese}
\newcommand{\overlineHsharp}{\overline{H}\diese}
\newcommand{\supperlb}[2]{\lfloor #1,#2\rfloor}
\newcommand{\dz}[1]{d z^{#1}}
\newcommand{\dzbar}[1]{d \bar{z}^{#1}}
\newcommand{\partialz}[1]{{\tfrac{\partial}{\partial{z^{#1}}}}}
\newcommand{\partialzbar}[1]{{\tfrac{\partial}{\partial{\bar{z}^{#1}}}}}
\newcommand{\anyform}{\lambda}
\newcommand{\picoeff}[1]{\pi^{#1}}
\newcommand{\thetacoeff}[2]{\theta^{#1}_{#2}}
\newcommand{\omegacoeff}[1]{\omega_{#1}}

\newcommand{\cover}{\mathcal{U}}

\newcommand{\algebroid}{B}

\newcommand{\anchorB}{\anchor_{\algebroid}}
\newcommand{\QB}{Q_{\algebroid}}
\newcommand{\QBhalf}{\QB^{\half}}
\newcommand{\dL}{\mathrm{d}_{\algebroid}}
\newcommand{\dLNonzero}{\mathrm{\widetilde{d}}_{\algebroid}}

\newcommand{\ssheaf}[1]{\mathcal{S}^{#1} }
\newcommand{\sxsheaf}[1]{\mathcal{\widetilde{S}}^{#1} }
\newcommand{\schech}[2]{\check{C}^{#1}(\cover;{\ssheaf{#2}})}
\newcommand{\sxchech}[2]{\check{C}^{#1}(\cover;{\sxsheaf{#2}})}

\newcommand{\cechdelta}{\delta}
\newcommand{\cohomology}[1]{\mathrm{H}^{#1}}
\newcommand{\tildecohomology}[1]{\mathrm{\widetilde{H}}^{#1}}
\newcommand{\homology}[1]{\mathrm{H}_{#1}}
\renewcommand{\AA}{\mathscr{A}}

\newcommand{\MC}{Maurer-Cartan }
\newcommand{\LP}{Lichnerowicz-Poisson }
\newcommand{\KB}{Koszul-Brylinski }
\newcommand{\Cech}{\v{C}ech }
\newcommand{\hFF}{\mathcal{O}_{\bX}}

\hypersetup{pdfpagemode=UseOutlines,colorlinks=false,pdfpagelayout=SinglePage,pdfstartview=FitH,bookmarksopen=true}

\begin{document}

\title[M-C Elements on Complex Manifolds]{Geometry of Maurer-Cartan Elements \\ on Complex Manifolds}
\author[Chen]{Zhuo Chen}
\thanks{Research partially supported by NSFC grant 10871007.}
\address{Tsinghua University, Department of Mathematics}
\email{\href{mailto:zchen@math.tsinghua.edu.cn}{\texttt{zchen@math.tsinghua.edu.cn}}}
\author[Sti\'enon]{Mathieu Sti\'enon}
\address{Universit\'e Paris Diderot, Institut de Math\'ematiques de Jussieu}
\email{\href{mailto:stienon@math.jussieu.fr}{\texttt{stienon@math.jussieu.fr}}}
\author[Xu]{Ping Xu}
\thanks{Research partially supported by NSF grants DMS-0605725 and DMS-0801129.}
\address{Pennsylvania State University, Department of Mathematics}
\email{\href{mailto:ping@math.psu.edu}{\texttt{ping@math.psu.edu}}}

\begin{abstract}
The semi-classical data attached to stacks of algebroids in the
sense of Kashiwara and Kontsevich are \MC elements on complex
manifolds, which we call extended Poisson structures as they
generalize holomorphic Poisson structures. A canonical Lie algebroid
is associated to each \MC element. We study the geometry underlying
these \MC elements in the light of Lie algebroid theory. In
particular, we extend \LP cohomology and \KB homology to the realm
of extended Poisson manifolds; we establish a sufficient criterion
for these to be finite dimensional; we describe how homology and
cohomology are related through the Evens-Lu-Weinstein duality
module; and we describe a duality on \KB homology, which generalizes
the Serre duality of Dolbeault cohomology.
\end{abstract}

\maketitle

\tableofcontents

\section{Introduction}
Due to their close connection to mirror symmetry, noncommutative 
deformations of complex manifolds have recently generated
 increasing interest~\cites{MR1855264,Bondal}.
Kashiwara-Kontsevich's stacks of algebroids are one way of substantiating the abstract concept of quantum complex manifolds (or noncommutative deformations of complex manifolds)~\cites{MR1855264,MR1384750,KS1,KS2,MR2348030,Tsygan1,Y1}.
The quantization of the sheaf of holomorphic functions $\hFF$ of a complex manifold $\bX$ may no longer produce a sheaf of algebras but, instead, 
lead to a nonabelian gerbe over the complex manifold
 $\bX$~\cites{MR2348030,Y2} or, in Kontsevich's terminology, a stack
 of algebroids.
Roughly speaking, an algebroid \textit{\`a la} Kontsevich consists of an open cover $\{U_i\}_{i\in I}$ of the complex manifold $\bX$,
a sheaf of associative unital algebras $\AA_i$ on each $U_i$,
 an isomorphism of algebras $g_{ij}:\AA_j|_{U_{ij}}\to\AA_i|_{U_{ij}}$
 for each nonempty intersection $U_{ij}$, and an invertible element
 $a_{ijk}\in \Gamma (U_{ijk}, \AA_i^{\times} )$
 for each triple intersection
 $U_{ijk}$. The isomorphisms $g_{ij}$ do not satisfy the usual cocycle condition. Instead, the equations $g_{ij}\circ g_{jk}\circ g_{ki}=\Ad_{a^{-1}_{ijk}}$ are satisfied as well as other compatibility conditions (among which a ``tetrahedron equation'').
In the terminology of~\cite{LSX:adv}, an algebroid \textit{\`a la} Kontsevich would be described as an extension of a \Cech groupoid by algebras.
A stack of algebroids can be thought of as a Morita equivalence class (see~\cite{LSX:adv}) of algebroids.
A canonical abelian category of coherent sheaves can be defined on a quantum complex manifold using its stack of algebroids description~\cites{MR1855264,MR1384750,KS1,KS2}.

It is well known that the semi-classical data attached to quantum real manifolds (i.e. star-algebras) are Poisson structures~\cites{MR0496157,MR0496158}. The cotangent bundle of a real Poisson manifold $(M,\pi)$ is endowed with a canonical Lie algebroid structure denoted by $(T^*M)_{\pi}$. This Lie algebroid structure plays a central role in Poisson geometry.
For instance, the Lichnerowicz-Poisson cohomology is simply the Lie algebroid cohomology of $(T^*M)_{\pi}$ with trivial coefficients.
Evens-Lu-Weinstein discovered a procedure for constructing a canonical module over a given Lie algebroid.
With the canonical module of $(T^*M)_{\pi}$ at hand, they interpreted Koszul-Brylinski homology as a Lie algebroid cohomology.  According to Kontsevich's formality theorem and Tsygan's chain formality theorem, the Hochschild cohomology and Hochschild homology of a star algebra are isomorphic to the Lichnerowicz-Poisson cohomology and Koszul-Brylinski homology of the underlying Poisson manifold.

In the context of complex geometry, the semiclassical data associated to
quantum complex manifolds are solutions of the Maurer-Cartan equation
in the derived global sections $R\sections{ X,\wedge^\bullet \TX[1]}$
of the sheaf of graded Lie algebras $\wedge^\bullet \TX[1]$ of polyvector fields on $\bX$,
which, according to Kontsevich's formality theorem, classify the deformations
of stacks of algebroids up to gauge transformations~\cites{MR1855264,Y1,MR2348030}.
More precisely, a Maurer-Cartan element is an
\[ H=\pi+\theta+\omega
\in\formss{0,0}(\wedge^2 \tangent{1,0})\oplus\formss{0,1}(\wedge^1
\tangent{1,0})\oplus\formss{0,2}(\wedge^0 \tangent{1,0})
\] (where $\formss{0,p}(\wedge^q
\tangent{1,0}$) denotes the space of $\wedge^q\tangent{1,0}$-valued
$(0,p)$-forms on $\bX$) satisfying the following equations:
\begin{align*}
&\barpd \omega+\ba{\omega}{\theta}=0,  &\barpd \pi + \ba{\theta}{\pi}=0,  \\
&\barpd \theta+\ba{\omega}{\pi}+\thalf\ba{\theta}{\theta}=0,  &\ba{\pi}{\pi}=0.
\end{align*}

Holomorphic Poisson bivector fields are special cases of such \MC elements, as are holomorphic $(0,2)$-forms.
For this reason, complex manifolds endowed with such a
\MC element $H$ will be called extended Poisson manifolds. In a recent paper~\cite{arXiv:0903.5065}, one of the authors studied the Koszul-Brylinski homology of holomorphic Poisson manifolds, and established a duality on it using the general theory developed by Evens-Lu-Weinstein~\cite{MR1726784}.

In this paper, in order to study the geometry of extended Poisson
manifolds, we apply the Evens-Lu-Weinstein theory to complex Lie
algebroids. Indeed, considering Maurer-Cartan elements as
Hamiltonian operators (in the sense of~\cite{MR1472888}) deforming a
Lie bialgebroid~\cite{MR1262213}, we define a complex Lie algebroid,
which mimics the role played by the cotangent Lie algebroid in real
Poisson geometry. It is not surprising that, for a holomorphic
Poisson structure, this complex Lie algebroid is the derived Lie
algebroid of the holomorphic cotangent Lie algebroid $(\TsX)_{\pi}$,
i.e. the matched pair $\Toi\bowtie (\TsX)_{\pi}^{(1,0)}$ studied
in~\cites{MR2439547,arXiv:0903.5065}. Using this complex Lie
algebroid, we introduce a Lichnerowicz-Poisson cohomology and a
Koszul-Brylinski homology for extended Poisson manifolds, and study
the relation between them. We extend the notion of coisotropic
submanifolds of holomorphic Poisson manifolds to the ``extended''
setting. We give a criterion on the ellipticity
of the complex Lie algebroid (in the sense of Block~\cite{math/0509284})
  induced by a \MC element. 
And in the elliptic case, we obtain a duality, which we
call Evens-Lu-Weinstein duality, on the Koszul-Brylinski homology
groups. As was pointed out in~\cite{arXiv:0903.5065} for the
holomorphic Poisson case, this duality generalizes the Serre duality
on Dolbeault cohomology.

Note that, modulo gauge equivalences, our extended Poisson structures and
 Yekutieli's Poisson deformations (see~\cite{Y1}) are equivalent.
 It  would be interesting to explore the connection between our results
 on Poisson homology and Berest-Etingof-Ginzburg's~\cite{MR2034924}. It
 would  also  be  interesting to
investigate if one can extend the method in this paper
to study the  Bruhat-Poisson structures of Evens-Lu on
 flag varieties  \cite{EL:flag}
and the toric Poisson structures of Caine \cite{Caine}.

\textbf{Acknowledgments.}
We would like to thank Penn State University (Chen),  ETH~Zurich (Xu) and 
Peking University (Sti\'enon and Xu) for their hospitality while work on this project was being done.
We also wish to thank many people for useful discussions and comments, including Camille Laurent-Gengoux, Giovanni Felder, Jiang-Hua Lu, Pierre Schapira and Alan Weinstein.

\section{Preliminaries}

\subsection{Lie bialgebroids}

A complex Lie algebroid~\cite{arXiv:math/0601752} consists of a complex vector bundle $\lad\to M$, 
a bundle map
$\anchor:\lad\to T_\CC M$ called anchor, and a Lie algebra bracket
$\lb{\cdot}{\cdot}$ on the space of sections $\sections{\lad}$ such
that $\anchor$ induces a Lie algebra homomorphism from
$\sections{\lad}$ to $\XX_\CC(M)$ and the Leibniz rule
\begin{equation*}\label{1} \lb{u}{fv}=\big(\anchor(u)f\big)v+f\lb{u}{v} \end{equation*}
is satisfied for all $f\in\cinf{M,\CC}$ and $u,v\in\sections{\lad}$.

It is well-known that a Lie algebroid
$(\lad,\ba{\cdot}{\cdot},\anchor)$ is equivalent to a Gerstenhaber algebra
$(\sections{\wedge\graded\lad},\wedge,\lb{\cdot}{\cdot})$~\cite{MR1675117}.
On the other hand, for a Lie algebroid structure on a vector bundle $\lad$,
there is also a degree~1 derivation $\dee$ of the graded commutative algebra
$(\sections{\wedge\graded\dal},\wedge)$ such that $\dee^2=0$.
The differential $\dee$ is given by
\begin{multline*}
\label{m2} (\dee\alpha)(u_0,u_1,\cdots,u_n)=\sum_{i=0}^n (-1)^i
\anchor(u_i) \alpha(u_0,\cdots,\widehat{u_i},\cdots,u_n) \\
+ \sum_{i<j} (-1)^{i+j}
\alpha(\ba{u_i}{u_j},u_0,\cdots,\widehat{u_i},\cdots,\widehat{u_j},\cdots,u_n).
\end{multline*}
Indeed, a Lie algebroid structure on $\lad$ is also  equivalent to
a differential graded algebra $(\sections{\wedge\graded\dal},\wedge,\dee)$.

Let $\lad\to M$ be a complex vector bundle. Assume that $\lad$ and its dual $\dal$
both carry Lie algebroid structures with anchor maps $\anchor:\lad\to
T_\CC M$ and $\anchors:\dal\to T_\CC M$, brackets on sections
$\sections{\lad}\otimes_{\CC}\sections{\lad}\to\sections{\lad}:u\otimes
v\mapsto\ba{u}{v}$ and
$\sections{\dal}\otimes_{\CC}\sections{\dal}\to\sections{\dal}:\alpha\otimes
\beta\mapsto\bas{\alpha}{\beta}$, and differentials
$\dee:\sections{\wedge^{\bullet}\dal}\to\sections{\wedge^{\bullet+1}\dal}$
and
$\dees:\sections{\wedge^{\bullet}\lad}\to\sections{\wedge^{\bullet+1}\lad}$.

This pair of Lie algebroids $(\lad,\dal)$ is a Lie bialgebroid
~\cites{MR1362125,MR1746902,MR1262213} if $\dees$ is a
derivation of the Gerstenhaber algebra $(\sections{\wedge\graded
\lad},\wedge,\ba{\cdot}{\cdot})$ or, equivalently, if $\dee$ is a
derivation of the Gerstenhaber algebra $(\sections{\wedge\graded
\dal},\wedge,\bas{\cdot}{\cdot})$. Since the bracket
$\bas{\cdot}{\cdot}$ (resp.\ $\ba{\cdot}{\cdot}$) can be recovered
from the derivation $\dees$ (resp.\ $\dee$), one is led to the
following alternative definition.

\begin{prop}[\cite{MR1675117}]
A Lie bialgebroid $(\lad,\dal)$ is equivalent to a differential Gerstenhaber algebra
structure on $(\sections{\wedge\graded\lad},\wedge,\ba{\cdot}{\cdot},\dees)$
(or, equivalently, on $(\sections{\wedge\graded \dal},\wedge,\ba{\cdot}{\cdot}_*,\dee)$).
\end{prop}

\subsection{Hamiltonian operators}
Let $(\lad,\dal)$ be  a complex Lie bialgebroid, and $H\in\sections{\wedge^2\lad}$.
We now replace the differential
$\dees:\sections{\wedge\graded\lad}\to\sections{\wedge^{\bullet+1}\lad}$
by a twist by $H$:
\begin{equation}
\label{eq:dH}
\deesH:\sections{\wedge^\bullet\lad}\to\sections{\wedge^{\bullet+1}\lad}, \qquad \deesH u=\dees u+\ba{H}{u}.
\end{equation}
It follows from a simple verification that if $H$ satisfies the Maurer-Cartan equation:
\begin{equation}
\label{eq:Maure} d_{*}H +\thalf [H,H]=0,
\end{equation}
then $(\deesH)^2=0$ and
$(\sections{\wedge\graded\lad},\wedge,\ba{\cdot}{\cdot},\deesH)$
is again a differential Gerstenhaber algebra. Thus
one obtains a Lie bialgebroid $(\lad,\dal_H)$.
A solution $H\in\sections{\wedge^{2}\lad}$ to Eq.~\eqref{eq:Maure} is called a
\textbf{Hamiltonian operator}~\cite{MR1472888}.
The Lie algebroid structure on $\dal_H$ can be described
explicitly: the anchor and the Lie bracket are given, respectively, by
\[ \anchorsH=\anchors+\anchor\circ\Hsharp \] and
\[ \basH{\alpha}{\beta} =\bas{\alpha}{\beta}+[\alpha,\beta]_{H} .\]
Here
\[ [\alpha,\beta]_{H} =\LieDer_{\Hsharp(\alpha)}\beta
-\LieDer_{\Hsharp(\beta)}\alpha-\dees \duality{\Hsharp(\alpha )}{\beta} ,\]
for all $\alpha,\beta\in\sections{\dal}$. We shall use $\dal_H$
to denote such a Lie algebroid and call it the $H$-twisted
Lie algebroid of $\dal$.
Thus we obtain the following
theorem, which was first proved in~\cite{MR1472888}
by a different method.

\begin{thm}
If $(\lad,\dal)$ constitutes a Lie bialgebroid, and
$H\in\sections{\wedge^2\lad}$ is a Hamiltonian operator,
then $(\lad,\dal_H)$ is a Lie bialgebroid.
\end{thm}

\section{Maurer-Cartan elements}

\subsection{The Lie bialgebroid stemming from a complex manifold}

We fix a complex manifold $\bX$ of complex dimension $n$
with almost complex structure $J$. We regard the tangent bundle
$\TX$ as a real vector bundle over $\bX$. The complexification of $\TX$
is denoted $\TCX$, namely: $\TCX=\TX\otimes\CC$. Similarly,
$\TsCX=\TsX\otimes\CC$. Let $\JJ:\TCX\to\TCX$ be the
$\CC$-linear extension of the almost complex structure $J$, and
$\Tio$ and $\Toi$ its $+\SqrtMinusOne$ and $-\SqrtMinusOne$
eigenbundles, respectively.
We adopt the following notations:
\[ \tangent{p,q}=\wedge^p\Tio\otimes\wedge^q\Toi ,\]
\[ \cotangent{p,q}=\wedge^p\Tios\otimes\wedge^q\Tois .\]

Consider the following two vector bundles which are obviously
mutually dual:
\begin{equation}\label{Eqt:defnAbdAbds}
\lad=\Tio\oplus\Tois ,\quad \dal=\Toi\oplus\Tios .\end{equation}
We can endow $A$ with a complex Lie algebroid structure. 
The anchor is the projection onto the first component: 
\[ a(\partialz{i})=\partialz{i} \qquad a(d\cc{z_j})=0 .\] 
The bracket of two sections of $\Tio$ is their bracket as vector fields; 
the bracket of any pair of sections of $\Tois$ is zero; and the bracket of a holomorphic vector field (i.e. a holomorphic section of the holomorphic vector bundle $\Tio$) and an anti-holomorphic 1-form (i.e. an anti-holomorphic section of the holomorphic vector bundle $\Tois$) is also zero. 
Thus \[ [\partialz{i},\partialz{j}]=0, \qquad [d\cc{z_i},d\cc{z_j}]=0, \quad \text{and} \quad [\partialz{i},d\cc{z_j}]=0 .\] 
Together with the Leibniz rule, the above three rules completely determine 
the bracket of any two arbitrary sections of $A$.
Similarly, one endows $\dal$
with a complex Lie algebroid structure as well.
It is simple to see that $(\lad,\dal)$ constitutes a Lie bialgebroid.
Indeed $A$ and $A^*$ are transversal Dirac structures of the
Courant algebroid $\TCX\oplus\TsCX$, for they are the eigenbundles 
of the generalized complex structure on $\bX$ induced by its complex manifold structure~\cites{MR2013140,GualtieriThesis}. In the sequel we will use the symbols 
\begin{equation} \Tio\bowtie\Tois \quad \text{and} \quad \Toi\bowtie\Tios \end{equation}
to refer to $\lad$ and $\dal$ when seen as Lie algebroids~\cite{MR2439547}. 

Moreover, one has
\begin{gather*}
\wedge^k \lad\cong \bigoplus_{i+j=k} \tangent{i,0} \otimes \cotangent{0,j} ,\\
\wedge^k \dal\cong \bigoplus_{i+j=k} \tangent{0,i}\otimes \cotangent{j,0} .
\end{gather*}

The Lie algebroid differentials associated to the Lie
algebroid structures on $\dal$ and $\lad$ are the
usual $\barpd$- and $\pd$-operators, respectively:
\begin{gather*}
\dees=\barpd:~ \formss{0,j} (\tangent{i,0}) \to \formss{0,j+1} (\tangent{i,0}) ,\\
\dee=\pd:~ \formss{j,0} (\tangent{0, i}) \to \formss{j+1,0} (\tangent{0, i}) .
\end{gather*}

\subsection{Extended Poisson structures}

\begin{defn}
An \textbf{extended Poisson manifold} $(\bX,H)$ is a complex manifold
$\bX$ equipped with an $H\in\sections{\wedge^2\lad}$ which is an
Hamiltonian operator with respect to $(\lad,\dal)$, i.e.
\begin{equation}\label{Eqt:ExtendedMC}
\barpd H+\thalf\ba{H}{H}=0 .\end{equation}
In this case, $H$ is called an extended Poisson structure.
\end{defn}

Any $H\in\sections{\wedge^2\lad}$ decomposes as
\[ H=\pi+\theta+\omega ,\] where
$\pi\in\sections{\tangent{2,0}}$,
$\theta\in\sections{\Tio\otimes \Tois}$
and $\omega\in\sections{\cotangent{0,2}}$.
We will use the following notations to denote
the bundle maps induced by natural contraction:
\begin{align*}
& \theta\bemol:~\Toi\to\Tio ,\\
& \theta\diese:~\Tios\to\Tois ,\\
& \pi\diese:~\Tios\to\Tio ,\\
& \omega\bemol:~\Toi\to\Tois .
\end{align*}
Note that $\theta\diese=-(\theta\bemol)^*$.

The following lemma is immediate.

\begin{lem}
An element $H=\pi+\theta+\omega$ is an extended Poisson structure 
if and only if the following equations are satisfied:
\begin{align}
&\barpd \omega+\ba{\omega}{\theta}=0, \label{eq:omega}\\
&\barpd \theta+\ba{\omega}{\pi}+\thalf\ba{\theta}{\theta}=0, \label{eq:theta}\\
&\barpd \pi + \ba{\theta}{\pi}=0, \label{eq:pi1}\\
&\ba{\pi}{\pi}=0. \label{eq:pi2}
\end{align}
\end{lem}\label{Lem:MCEqtSpellOut}

\begin{rmk}

When only one of the three terms of $H$ is not zero, 
we are left with one of the following three
special cases:
\begin{enumerate}
\item $H=\pi$ is an extended Poisson if and only if
 $\pi$ is a
holomorphic Poisson bivector field.
\item $H=\theta$ is an extended Poisson  if and only if
$\barpd\theta+\thalf\ba{\theta}{\theta}=0$. Moreover, if
$\overlinethetaflat\circ\theta\bemol-\id$ is invertible,
$\theta$ is equivalent to a deformed complex structure
\cite{MR2109686}.
\item $H=\omega$ is an extended Poisson if and only if $\barpd\omega=0$.
\end{enumerate}

In fact, if $\ba{\omega}{\pi}=0$, Eq.~\eqref{eq:theta} implies that
$\theta$ defines a deformed complex structure (under the assumption
that $\overlinethetaflat\circ\theta\bemol-\id$ is invertible).
Then, according to Lemma~\ref{Zurich} below,
Eq.~\eqref{eq:omega} is equivalent to
$\bar{\partial}_\theta\omega=0$,
where $\barpd_\theta=\barpd+[\theta,\cdot]$, 
and Eqs.~\eqref{eq:pi1}-\eqref{eq:pi2} mean 
that $\pi$ is a holomorphic Poisson tensor with
respect to the deformed complex structure.
\end{rmk}

\begin{cor}\label{Cor:Hminus}
If $H=\pi+\theta+\omega$ is an extended Poisson structure, then so is
\[ \lambda\pi+\theta+\lambda\inv\omega ,\] for any $\lambda\in\CC^\times$.
In particular, \[ H^{\vee}=-\pi+\theta-\omega \] is an extended Poisson structure.
\end{cor}

Note that Maurer-Cartan elements as deformations of Lie bialgebroids
or differential Gerstenhaber algebras were already considered by Cleyton-Poon~\cite{CP} 
in their study of nilpotent complex structures on real six-dimensional nilpotent algebras.

A natural question is: when will $(\lad,\dal_H)$ arise from a
generalized complex structure in the sense of Hitchin
\cites{MR2013140,GualtieriThesis}? Let us recall the following:
\begin{lem}(Lemma 6.1 in~\cite{arxivmath0702718})
The graph $\{H\diese\xi+\xi\in\lad\oplus \dal\}$ of $H$, 
which is clearly isomorphic to $\dal_H$ as a vector bundle, 
is the $+\SqrtMinusOne$- (or $-\SqrtMinusOne$-) eigenbundle 
of a generalized complex structure on $\bX$ if and only if
$\overlineHsharp \circ \Hsharp-\Id{\dal}$ is invertible. 
Here the map $\overlineHsharp:\lad\to\dal$ is defined by 
$\overlineHsharp(u)=\overline{\Hsharp(\overline{u})}$, $\forall u\in\lad$.
\end{lem}

Again we let $H=\pi+\theta+\omega$ be an extended Poisson structure
on $\bX$. Relatively to the direct sum decompositions \label{eq3} of $\lad$ and $\dal$, the endomorphisms $\Hsharp$ and $\overlineHsharp$ 
are represented by the block matrices
\[ \Hsharp= \left( \begin{matrix} \theta\bemol & \pi\diese \\ 
\omega\bemol & \theta\diese \end{matrix} \right) 
\quad \text{and} \quad 
\overlineHsharp= \left( \begin{matrix} \overlinethetaflat & \overlinepisharp \\ 
\overlineomegaflat & \overlinethetasharp \end{matrix} \right) .\]
In turn, we have
\begin{equation}
\label{Eqt:overlineHsharpcircHsharp}
\overlineHsharp  \Hsharp=\left(
\begin{matrix}
  \overlinethetaflat \circ\theta\bemol+ \overlinepisharp\circ\omega\bemol
  & \overlinethetaflat\circ \pi\diese+\overlinepisharp\circ\theta\diese \\
  \overlineomegaflat\circ \theta\bemol +\overlinethetasharp\circ
  \omega\bemol  & \overlineomegaflat\circ\pi\diese+
  \overlinethetasharp\circ \theta\diese \end{matrix} \right).
\end{equation}

\begin{prop} Given an extended Poisson manifold $(\bX,H)$, let $\lad=\Tio\bowtie\Tois$. 
Then $\dal_H$ is the $(\pm\SqrtMinusOne)$-eigenbundle of a generalized
complex structure if and only if $\overlineHsharp\Hsharp-\Id{\dal}$ is invertible.
\end{prop}

\begin{ex}
If $H=\pi$ (i.e. H is a holomorphic Poisson bivector field) or $H=\omega$, 
it is clear that $\overlineHsharp\Hsharp$ is zero. Hence, in these two situations, 
the extended Poisson structure on $\bX$ is actually a generalized complex structure.
\end{ex}

Here is a simple example of extended Poisson structure, 
which does not arise from a generalized complex structure.
\begin{ex}\label{Ex:realtorus} 
Consider the torus $\mathbf{T}=\CC/(\ZZ+\SqrtMinusOne\ZZ)$ with its standard complex structure. 
Let $z$ be the standard coordinate on $\mathbf{T}$. Obviously, any
\begin{equation}\label{Eq:thetafddzdbarz}
\theta=f(z,\bar{z})\tfrac{d}{dz}\wedge d\bar{z},
\end{equation} 
where $f$ is a smooth $\CC$-valued function, 
is an extended Poisson structure.
In this case, $\overlineHsharp\Hsharp=\abs{f}^2\Id{}$.
Hence $\dal_\theta$ does not stem from a generalized complex structure provided that $\abs{f}=1$.
\end{ex}

\subsection{Elliptic Lie algebroids}
As in~\cite{math/0509284}, we say that a complex Lie algebroid $\algebroid$ is \textbf{elliptic} 
if $\RealPart\circ\anchorB:\algebroid\to\TX$ is surjective. Here $\anchorB:\algebroid\to\TCX$ is
the anchor map of $\algebroid$ and $\RealPart:~\TCX\to\TX$ is the projection onto the real part.

\begin{thm}[\cite{math/0509284}]\label{extremely}
If $\algebroid$ is  an elliptic Lie algebroid over a compact complex
manifold $\bX$, and $E$ a finite rank complex vector
bundle with a $B$-action as in \cite{MR1726784},
then all cohomology groups
$\cohomology{\bullet}(B,E)$ are finite dimensional.
\end{thm}

It is therefore natural to ask when $\dal_H$ is elliptic.
An easy calculation shows the following:

\begin{prop}\label{Vienna} 
Let $a_*^H$ denote the anchor of $\dal_H$ and $\conjugate:\Toi\to\Tio$ 
the complex conjugation. 
The bundle maps $\RealPart\rond a_*^H$ and 
\begin{equation}\label{Paris}
\altmap=(\conjugate+\theta\bemol)\oplus\pi\diese :\Toi\oplus\Tios\to\Tio ,
\end{equation}
and the isomorphism of real vector bundles $\RealPart:\Tio\to\TX$ fit into the commutative diagram 
\begin{equation}\label{Diagram:alternateMapanchorRealCommute}
\xymatrix{ & \Toi \oplus \Tios \ar[dl]_{\altmap}\ar[dr]^{\RealPart\rond a_*^H} \\
\Tio \ar[rr]^{\RealPart} & & \TX. } \end{equation}
As a consequence, $\dal_H$ is an elliptic Lie algebroid
if and only if $\altmap$ is surjective.
\end{prop}

\begin{ex} 
When $H=\pi$, or $\omega$, it is clear that $\dal_H$ is elliptic. On the other hand, if we consider the torus $\bT$ endowed with the bivector field $\theta$ of Example~\ref{Ex:realtorus}, the Lie algebroid $\dal_H$ is elliptic if and only if $f$ is not identically $1$.
\end{ex}

\subsection{Poisson cohomology}

\begin{defn}
Given an extended Poisson manifold $(\bX,H)$, the cohomology of
the Lie algebroid $\dal_H$ is called the \textbf{Poisson cohomology}
of the extended Poisson structure, and denoted $\cohomology{\bullet}
(\bX,H)$. In other words, it is the cohomology of the cochain complex:
\begin{equation}\label{eq:H}
\cdots \xto{\barpdH} \sections{\wedge^{k}\lad} \xto{\barpdH} \sections{\wedge^{k+1}\lad} \xto{\barpdH} \cdots
,\end{equation}
where $\sections{\wedge^k\lad}=\oplus_{i+j=k}\formss{0, j}(\tangent{i,0})$ and $\barpdH=\barpd+\ba{H}{\cdot}$.
\end{defn}

Poisson cohomology is also called tangent cohomology by Kontsevich~\cite{MR2062626}.

As an immediate consequence of Theorem~\ref{extremely} and Proposition~\ref{Vienna}, we have
\begin{cor}
If $H$ is an extended Poisson structure on a compact complex
manifold $\bX$ and the map $\altmap$ (given by Eq.~\eqref{Paris})
is surjective, then all Poisson cohomology groups are finite
dimensional.
\end{cor}


\begin{rmk}
When $H$ is a holomorphic Poisson bivector field $\pi$, the cochain
complex~\eqref{eq:H} is the total complex of the double complex as
discussed in Corollary~4.26 in~\cite{MR2439547}.
\end{rmk}

On the other hand, if $H=\theta\in\formss{0,1}(\tangent{1,0})$ is a
Maurer-Cartan element such that $\overlinethetaflat\circ\theta\bemol-\id$
is invertible, then $\theta$ defines a
new complex structure on $\bX $ according to Kodaira~\cite{MR2109686}.

The following lemma can be verified directly.

\begin{lem}
\label{Zurich} Let $H=\theta\in\formss{0,1}(\tangent{1,0})$
be a Maurer-Cartan element such that
$\overlinethetaflat \circ\theta\bemol-\Id{}$ is invertible.
Then the Lie algebroid $\dal_H$ is isomorphic to
$T^{1,0}_\theta\bX\bowtie(T^{0,1}_\theta\bX)^*$,
where $T^{1,0}_\theta\bX$ and $T^{0,1}_\theta\bX$ are, respectively,
the $+\SqrtMinusOne$ and $-\SqrtMinusOne$ eigenbundles of the deformed almost complex structure
$J_\theta:\TX\to\TX$. As a consequence, the differential operator
$\deesH$ in Eq.~\eqref{eq:dH} is equal to $\barpd_\theta$, the new
$\barpd$-operator of the deformed complex structure.
\end{lem}

Thus we have
\begin{prop}
If $H=\theta\in\formss{0,1}(\tangent{1,0})$ is a Maurer-Cartan element such that
$\overlinethetaflat\circ\theta\bemol-\Id{}$ is invertible, then
\[ \cohomology{k}(\bX,H)\cong\oplus_{i+j=k}\cohomology{i}(\bX,\wedge^j T_\theta\bX) ,\]
where $T_\theta\bX$ denotes the holomorphic tangent bundle of
the deformed complex manifold $\bX$.
\end{prop}

\subsection{Coisotropic submanifolds}

Suppose that $\bY\subseteq \bX$ 
is a complex submanifold~\cite{MR2109686}. Set
\begin{gather*} 
\NYio=\set{\xi\in(\Tio|_{\bY})^* \st \duality{\xi}{Y}=0,\;\forall Y\in\TYio} ,
\end{gather*}
and consider the subbundle $\KKK=\TYoi\oplus\NYio$ of $\dal$. 

\begin{defn} 
A complex submanifold $\bY$ of $\bX$ is called \textbf{coisotropic} if
$H(u,v)=0$, for all $u,v\in\KKK$.
\end{defn}

\begin{ex} 
If $H=\pi$ is a holomorphic Poisson bivector field, 
then $\bY$ is coisotropic if and only if it is coisotropic in the usual sense,
i.e. ${\pi}{(\xi_1,\xi_2)}=0$, $\forall\xi_1,\xi_2\in\NYio$, 
or $\pi\diese(\NYio)\subseteq\TYio$.
\end{ex}

\begin{ex} 
If $H=\omega$, then $\bY$ is coisotropic if and only if
$\iota^*\omega=0$, where $\iota:\bY\to\bX$ is the embedding map.
\end{ex}

\begin{ex} 
If $H=\theta$, then $\bY$ is coisotropic if and only if $\theta\bemol(\TYoi)\subseteq\TYio$.
\end{ex}

It is well known that given a coisotropic submanifold $C$ of a real Poisson manifold $(P,\pi)$, 
the conormal bundle $NC=\set{\xi\in T_c^*P \st c\in C;\; \duality{\xi}{X}=0,\;\forall X\in T_c C}$ 
is a Lie subalgebroid of the cotangent Lie algebroid $(T^*P)_\pi$~\cite{MR959095}.
The following proposition can be considered as an analogue of this fact in the extended Poisson setting.

\begin{prop}
Let $\bY$ be a coisotropic submanifold of the extended Poisson manifold $(\bX,H)$. 
Then the vector subbundle $\KKK=\TYoi\oplus\NYio$
is a Lie subalgebroid of $\dal_H$. That is, $a_*^H$ maps $\KKK$ into $T_{\CC}\bY$ 
and for any smooth extensions $\widetilde{u},\widetilde{v}\in\sections{\dal_H}$ to $\bX$ of any two sections $u,v\in\sections{\KKK}$, the restriction to $\bY$ of 
$\basH{\widetilde{u}}{\widetilde{v}}$ 
is a section of $\KKK$ which does not depend on the choice of extensions.
\end{prop}

\subsection{Poisson relations}

Following Weinstein~\cite{MR959095}, we introduce the following
\begin{defn}
Let $(\bX_1,H_1)$ and $(\bX_2,H_2)$ be extended Poisson manifolds.
A Poisson relation from $(\bX_2,H_2)$ to $(\bX_1,H_1)$
is a coisotropic submanifold
of the product manifold $\bX_1\times\bX_2^{\vee}$
(i.e. $\bX_1\times\bX_2$ endowed with the extended Poisson structure
 $(H_1, H_2^{\vee})$, 
see Corollary~\ref{Cor:Hminus}).
\end{defn}

We call a holomorphic map $f:\bX_2\to\bX_1$ between 
 extended Poisson manifolds $(\bX_1,H_1)$ and $(\bX_2,H_2)$ an \textbf{extended Poisson map} if its graph
\[ G_f=\set{(f(x),x) \st x\in\bX_2}\subset\bX_1\times\bX_2^{\vee} \]
is a Poisson relation.

\begin{prop} Let $(\bX_1,H_1)$ and $(\bX_2,H_2)$
be extended Poisson manifolds, where the extended Poisson structures decompose as $H_i=\pi_i+\theta_i+\omega_i$ ($i=1,2$). Then a holomorphic map
$f:\bX_2\to\bX_1$ is an extended Poisson map if and only if
$f_*\pi_2=\pi_1$; $f^*\omega_1=\omega_2$; and $f_*\circ\theta\bemol_2=\theta\bemol_1\circ f_*$.
\end{prop}

The proof is a direct verification and is left to the reader.
As a consequence, we have
\begin{cor}The composition of two extended Poisson maps is again an
extended Poisson map.
\end{cor}


\section{Koszul-Brylinski Poisson homology}
In this section we will introduce homology groups for extended
Poisson manifolds based on the Evens-Lu-Weinstein
module of a Lie algebroid.

\subsection{Koszul-Brylinski cochain complex}
First we recall the notion of Clifford algebras and spin
representation.  Let $V$ be a vector space of dimension $n$ endowed
with a non-degenerate symmetric bilinear form $\ip{\cdot}{\cdot}$.
Its Clifford algebra $\clifford{V}$ is defined as the quotient of
the tensor algebra $\oplus_{k=0}^n V^{\otimes k}$ by the relations
$x\otimes y+y\otimes x=2\ip{x}{y}$, with $x,y\in V$. It is naturally an
associative $\ZZ_2$-graded algebra. Up to isomorphisms, there exists
a unique irreducible module $S$ of $\clifford{V}$ called spin
representation~\cite{MR1636473}. The vectors of $S$ are called
spinors.

An operator $O$ on $S$ is called even (or of degree~0) if
$O(S^i)\subset S^{i}$ and odd (or of degree~1) if $O(S^i)\subset
S^{i+1}$. Here $i\in\ZZ_2$. If $O_1$ and $O_2$ are operators of
degree~$d_1$ and $d_2$ respectively, then their commutator is the
operator \[ \supperlb{O_1}{O_2}= O_1\circ O_2-(-1)^{d_1 d_2}O_2\circ O_1 .\]

\begin{ex}\label{translagr}
Let $W$ be a vector space of dimension $r$. We can endow $V=W\oplus
W^*$ with the non-degenerate pairing
\begin{equation*}\label{5} \ip{u_1+\xi_1}{u_2+\xi_2}=\thalf\big(\xi_1(u_2)+
\xi_2(u_1)\big) ,\end{equation*}
where $u_1,u_2\in W$ and $\xi_1,\xi_2\in W^*$. The
representation of $\clifford{V}$ on $S=\oplus_{k=0}^r \wedge^k W$
defined by $u\cdot w=u\wedge w$ and $\xi\cdot w=\ii{\xi}w$, where
$u\in W$, $\xi\in W^*$ and $w\in S$, is the spin representation.
Note that $S$ is $\ZZ$- and thus also $\ZZ_2$-graded.
\end{ex}

Recall that $\EE=\TCX\oplus\TsCX$ admits the standard pseudo-metric
\[ \ip{X_1+\xi_1}{X_2+\xi_2}= \thalf\big(\duality{\xi_1}{
X_2}+\duality{\xi_2}{ X_1}\big) ,\]
where $X_i\in\TCX$ and $\xi_i\in\TsCX$.
The corresponding Clifford bundle $\clifford{\EE}$
can be identified with the vector bundle
$(\wedge^\bullet\TCX)\otimes(\wedge^\bullet\TsCX)$, under which
the Clifford action of $\clifford{\EE}$ on the spinor bundle
\[ \wedge^\bullet\TsCX=\bigoplus_{p,q}\cotangent{p,q} \]
is given by
\[ (W\otimes \xi )\cliffact \anyform=
(-1)^{\frac{w(w-1)}{2}}\ii{W}(\xi \wedge \anyform) .\]
Here $W\in \wedge^w \TCX$, $\xi , \anyform\in \wedge^\bullet
\TsCX$, and the symbol $\ii{W}$ denotes the standard contraction
\[ \duality{\ii{W}\xi }{X}=\duality{\xi }{W\wedge X} ,\]
for $\xi\in\wedge^p\TsCX$ and $X\in\wedge^{p-w}\TCX$ with $p\geq w$.

Let $(\bX,H)$ be an extended Poisson manifold of complex
dimension $n$. Then $\dal_H$ is
a Lie algebroid and the \textbf{Evens-Lu-Weinstein module}~\cite{MR1726784}
of $\dal_H$ is the complex line bundle
\[ \QAsH=\wedge^{2n}\dal_H\otimes\wedge^{2n}\TsCX .\]
The representation of $\dal_H$ on $\QAsH$ is given by
\begin{multline*}\label{Eqn:RepELW}
\representation_{\alpha}(\alpha_1\wedge\cdots\wedge\alpha_{2n}\otimes\mu)
=\sum_{i=1}^{2n}\big(\alpha_1\wedge\cdots\wedge\lbsH{\alpha}{\alpha_i}
\wedge\cdots\wedge\alpha_{2n}\otimes\mu\big) \\
+\alpha_1\wedge\cdots\wedge\alpha_{2n}\otimes\ld{\anchorsH(\alpha)}\mu,
\end{multline*}
where $\alpha,~\alpha_1,\cdots ,\alpha_{2n}\in\sections{\dal_H}$,
$\mu\in \sections{\wedge^{2n} \TsCX}$.

A simple computation yields that
$\QAsH\cong\wedge^{n}\Tios\otimes\wedge^{n}\Tios$.
Accordingly,  
\[ \module=\QAsHHalf\cong\wedge^n\Tios=\cotangent{n,0} \] 
is also an $\dal_H$-module and we use $\representation$ again to denote the representation. 
Equivalently, we have an operator 
\begin{equation} \label{Eqt:DDH} 
\DDH:~\sections{\module}\to\sections{\lad\otimes\module},
\end{equation} such that
\[ \ii{\alpha}\DDH s=\representation_{\alpha}s,\quad\forall\alpha\in\sections{\dal}, s\in\sections{\module} ,\]
which allows us to define a differential operator
\[ \bdeesH:\sections{\wedge^k\lad\otimes\module}\to\sections{\wedge^{k+1}\lad\otimes\module} \]
by
\begin{equation} \label{Eqt:bdeesHuotimess}
\bdeesH(u\otimes s)=(\barpdH u)\otimes s+(-1)^{k}u\wedge\DDH s,
\end{equation}
for all $u\in\sections{\wedge^k\lad}$ and $s\in\sections{\module}$.

The following lemma is needed later.

\begin{lem} The relation
\[ \stdiso(X \otimes s) = X\cliffact s ,\]
where in the r.h.s.\ $X\in\wedge^k\lad$ is regarded as an element of the Clifford algebra
$\clifford{\EE}$ and $s\in\module$ is regarded as an element in $\wedge^\bullet\TsCX$,
defines an isomorphism of vector bundles
\[ \stdiso:\wedge^k\lad\otimes\module\to\bigoplus_{i-j=n-k}\cotangent{i,j} .\]
\end{lem}
Equivalently,
\[ \stdiso\big((W\wedge\xi)\otimes s\big)=(-1)^{\frac{w(w-1)}{2}}\ii{W}(\xi\wedge s)
=(-1)^{\frac{w(w-1)}{2}+n(k-w)}(\ii{W}s)\wedge\xi ,\]
for $W\in\tangent{w,0}$, $\xi\in\cotangent{0, k-w}$ and $s\in \module $.

We define the inner product of
$H\in\sections{\wedge^2\lad}$ with $\anyform\in\sections{\wedge^\bullet\TsCX}$
as \[ \ii{H}\anyform= -H \cliffact \anyform .\]
This coincides with the usual inner product of bivector fields with
differential forms. Introduce
\[ \supperlb{\pd}{\ii{H}}= {\pd}\circ
{\ii{H}}-{\ii{H}}\circ {\pd}
:~\sections{\wedge^\bullet \TsCX}\to
\sections{\wedge^\bullet \TsCX} .\]

Let us denote $\forms{i,j}=\sections{\cotangent{i,j}}$. The
following theorem is the main result in this section.

\begin{thm}\label{Thm:main1}
The diagram
\begin{equation}\label{Diagram:stdIso}
\xymatrix{ \sections{\wedge^k\lad\otimes\module} \ar[rr]^{\bdeesH} \ar[d]_{\stdiso} & & 
\sections{\wedge^{k+1}\lad\otimes\module} \ar[d]^{\stdiso} \\ 
\bigoplus_{i-j=n-k} \forms{i,j} \ar[rr]_{\barpd+\supperlb{\pd}{\ii{H}}} & & \bigoplus_{i-j=n-k-1} \forms{i,j} }
\end{equation}
commutes.
\end{thm}

\begin{defn}
The cohomology of the cochain complex $(\bigoplus_{i-j=n-k}\forms{i,j},\barpd+\supperlb{\pd}{\ii{H}})$
is called the \textbf{Koszul-Brylinski Poisson homology} of the extended Poisson manifold $(\bX,H)$, and denoted $\homology{\bullet}(\bX,H)$.
\end{defn}

\begin{rmk}
\begin{enumerate}
\item
If $H=\pi$ is a holomorphic Poisson bivector field, the cochain complex 
$(\bigoplus_{i-j=n-k}\forms{i,j},\barpd+\supperlb{\pd}{\ii{H}})$ is the total
complex of a double complex. Its
cohomology is the usual Koszul-Brylinski Poisson homology
of a holomorphic  Poisson manifold, as studied in detail by
one of the authors~\cite{arXiv:0903.5065}.

\item
If $H=\omega\in\forms{0,2}$ with $\bar{\partial}\omega=0$, 
the complex $(\bigoplus_{i-j=n-k}\forms{i,j},\barpd+\supperlb{\pd}{\ii{H}})$ 
becomes $(\bigoplus_{i-j=n-k}\forms{i,j},\barpd+(\pd\omega)\wedge)$. 
Its cohomology is the twisted Dolbeault cohomology.

\item
If $H=\theta\in\Omega^{0,1}(\tangent{1,0})$ is a Maurer-Cartan element 
such that $\overlinethetaflat\circ\theta\bemol-\id$ is invertible, 
then $\theta$ defines a new complex structure on $\bX$. 
According to Lemma~\ref{Zurich}, the cochain complex 
$(\bigoplus_{i-j=n-k}\forms{i,j},\barpd+\supperlb{\pd}{\ii{H}})$ 
is isomorphic to $(\bigoplus_{i-j=n-k}\Omega_\theta^{i,j}(\bX),\;\bar{\partial}_\theta)$, 
where $\bar{\partial}_\theta$ is the $\bar{\partial}$-Dolbeault operator 
of the deformed complex structure. As a consequence, we have
$\homology{k}(\bX,\theta)\isomorphism\oplus_{j-i=n-k}\cohomology{i,j}_\theta(\bX)$, 
where $\cohomology{i,j}_\theta(\bX)$ is the Dolbeault cohomology 
of the deformed complex structure.
\end{enumerate}
\end{rmk}

\subsection{Evens-Lu-Weinstein duality}

Consider a compact complex (and therefore orientable) manifold
$\bX$ with $\dim_{\CC}\bX=n$, a complex Lie algebroid $B$
over $\bX$ with $\rk_{\CC}B=r$. According to~\cite{MR1726784},
the complex line bundle $\QB=\wedge^r B\otimes\wedge^{2n}\TsCX$
is a module over the complex Lie algebroid $B$. If $\QBhalf$ exists
as a complex vector bundle, $\QBhalf$ becomes a $B$-module as
well. There is a natural map
\[ \phi: \sections{\wedge^k B^*\otimes\QBhalf}
\otimes \sections{\wedge^{r-k} B^*\otimes\QBhalf} 
\to \sections{\wedge^r B^*\otimes\QB}
\isomorphism\sections{\wedge^{2n}T^*_\CC\bX} \]

Integrating, we get the pairing
\begin{equation} \label{eq:pairing} 
\sections{\wedge^k B^*\otimes\QBhalf} \otimes
\sections{\wedge^{r-k} B^*\otimes\QBhalf}\to\CC,  \qquad
\xi\otimes\eta\mapsto\int_{\bX}\phi(\xi\otimes\eta) 
.\end{equation}

The following result is essentially due to Evens-Lu-Weinstein~\cite{MR1726784} 
for the pairing, and to Block~\cite{math/0509284} for the non-degeneracy 
(see also~\cite{arXiv:0903.5065}).

\begin{thm} \label{Beijing} 
For a complex Lie algebroid $B$,
 with $\rk_{\CC}B=r$,  over a compact
manifold $\bX$, the pairing~\eqref{eq:pairing} induces a pairing
\[ \cohomology{k}(B,\QBhalf)\otimes\cohomology{r-k}(B,\QBhalf)\to\CC .\]
Moreover, if $B$ is an elliptic Lie algebroid, 
this pairing is non-degenerate.
\end{thm}

Let $(\bX,H)$ be a compact extended Poisson manifold of complex dimension $n$. 
Consider the Lie algebroid $B=(\Toi\bowtie\Tios)_H$.
Applying Theorem~\ref{Beijing} and Proposition~\ref{Vienna}, we obtain 
\begin{thm} \label{thm:1} 
Let $(\bX,H)$ be a compact extended Poisson manifold of complex dimension $n$, 
with $H=\pi+\theta+\omega$. 
Then the map \[ \forms{{i,j}}\otimes\forms{{k,l}}\to\CC:
\zeta\otimes\eta\mapsto\int_{\bX}(\zeta\wedge\eta)^{top} \]
induces a pairing on the Koszul-Brylinski Poisson homology:
\begin{equation} \label{eq:pairingPoisson} 
\homology{k}(\bX,H)\otimes\homology{2n-k}(\bX,H)\to\CC 
.\end{equation}
Moreover, if the bundle map
 $\altmap=(\conjugate+\theta\bemol)\oplus\pi\diese$ 
maps $\Toi\oplus\Tios$ surjectively onto $\Tio$, then all 
homology groups $\homology{\bullet}(\bX,H)$
are finite dimensional vector spaces 
and the pairing~\eqref{eq:pairingPoisson} is non-degenerate.
\end{thm}

\subsection{Proof of Theorem~\ref{Thm:main1}}
The following lemmas are needed.
\begin{lem}For any $u\in \sections{\wedge^p\lad}$, $\anyform\in
\forms{\cdot,\cdot}$, one has
\begin{equation}\label{Eqt:barpdulambda}
\barpd(u\cliffact\anyform)=(\barpd
u)\cliffact\anyform + (-1)^p u\cliffact
\barpd\anyform.
\end{equation}
\end{lem}
\begin{lem}For any $u\in\sections{\wedge^p\lad}$, $v\in
\sections{\wedge^q\lad}$, the Schouten bracket $\ba{u}{v}$ is
determined by
\begin{equation}\label{uvSchouten}
\ba{u}{v}\cliffact \anyform=(-1)^{q+1
}\supperlb{u}{\supperlb{v}{\pd}}
\anyform,\quad\forall \anyform\in \forms{\bullet,\bullet}.
\end{equation}
\end{lem}

Both lemmas can be proved by induction; this is left to the reader.

\begin{lem} \label{lem:pdiHulambdaproperty}
For any $u\in\sections{\wedge^i\lad}$ 
and $\anyform\in\forms{\bullet,\bullet }$, one has
\begin{equation} \label{Eqt:iHpdulambdaproperty}
\supperlb{\pd}{\ii{H}}(u\cliffact\anyform)
=\ba{H}{u}\cliffact\anyform
+(-1)^{i}u\cliffact(\supperlb{\pd}{\ii{H}}\anyform) 
.\end{equation}
In particular, for any smooth function $f\in\cinf{\bX,\CC}$, one has
\begin{equation} \label{Eqt:iHpdfproperty}
\supperlb{\pd}{\ii{H}}(f\anyform)
=\ba{H}{f}\cliffact\anyform
+f\supperlb{\pd}{\ii{H}}\anyform 
.\end{equation}
\end{lem}

\begin{proof}
According to Eq.~\eqref{uvSchouten}, we have
\begin{align*}
\ba{H}{u}\cliffact\anyform 
=& (-1)^{i+1} \supperlb{H}{\supperlb{u}{\pd}}\anyform \\
=& (-1)^i(u\cliffact\pd(H\cliffact\anyform)
-H\cliffact u\cliffact(\pd\anyform))
+(H\cliffact(\pd(u\cliffact\anyform))
-\pd(u\cliffact H\cliffact\anyform)) \\
=& (-1)^i (u\cliffact\pd(H\cliffact\anyform)
-u\cliffact H\cliffact(\pd\anyform))
+(H\cliffact(\pd(u\cliffact\anyform))
-\pd(H\cliffact u\cliffact\anyform)) \\
=& -(-1)^{i} u\cliffact(\supperlb{\pd}{\ii{H}}\anyform) +\supperlb{\pd}{\ii{H}}(u\cliffact\anyform) .\qedhere
\end{align*}
\end{proof}

A straightforward (though lengthy) computation shows the following:
\begin{lem} \label{Lem:localstrAbdsH}
Suppose that $(z^1,\dots,z^n)$ is a local
holomorphic chart and $H=\pi+\theta+\omega$ is given by
\begin{equation} \label{Eqt:Hlocally}
H= \picoeff{i,j}\partialz{i}\wedge\partialz{j}
+\thetacoeff{p}{q}\partialz{p}\wedge\dzbar{q} 
+\omegacoeff{k,l}\dzbar{k}\wedge\dzbar{l}
,\end{equation}
where $\picoeff{i,j}$, $\thetacoeff{p}{q}$, and $\omegacoeff{k,l}$ are 
complex valued smooth functions on $\bX$. Then the $H$-twisted Lie
algebroid structure on $\dal_H\cong\Toi\oplus\Tios$ can be expressed by:
\begin{gather}
\anchorsH(\partialzbar{i})=\partialzbar{i}-\thetacoeff{p}{i}\partialz{p},
\qquad \anchorsH(\dz{i})=2\picoeff{i,q}\partialz{q}, \label{align:anchorsHlocal} \\
\basH{\partialzbar{i}}{\partialzbar{j}}=2\pd\omegacoeff{i,j},
\quad \basH{\dz{i}}{\dz{j}}=2\pd \picoeff{i,j},
\quad \basH{\dz{j}}{\partialzbar{i}}=\pd\thetacoeff{j}{i}. \label{align:basHlocal}
\end{gather}
\end{lem}

\begin{lem}
Making the same assumptions as in Lemma~\ref{Lem:localstrAbdsH}, consider the local section 
\begin{equation} \label{Eqt:s} s=\dz{1}\wedge\cdots\wedge\dz{n} \end{equation}
of $\module=\QAsHHalf$. The representation of $\dal_H$ on $\module$ is given by
\begin{equation}\label{Eqt:RepHs}
\representation_{\partialzbar{i}}s=-\tfrac{\partial\thetacoeff{p}{i}}{\partial z^p}s ,
\qquad \representation_{\dz{i}}s=2\tfrac{\partial\picoeff{i,p}}{\partial z^p}s
.\end{equation}
\end{lem}

\begin{proof}
Using Eq.~\eqref{align:anchorsHlocal}, we compute
\begin{equation}\label{eqt:ldanchorsHonAbd}
\begin{aligned}
\ld{\anchorsH(\partialzbar{i})}{\dz{j}} &= -d\thetacoeff{j}{i}, &
\ld{\anchorsH(\partialzbar{i})}{\dzbar{j}} &= 0, \\
\ld{\anchorsH(\dz{i})}{\dz{j}} &= 2 d\picoeff{i,j}, &
\ld{\anchorsH(\dz{i})}{\dzbar{j}} &= 0.
\end{aligned}
\end{equation}

Write
\[ s^2=(\partialzbar{1}\wedge\cdots\wedge\partialzbar{n}\wedge\dz{1}\wedge\cdots\wedge\dz{n})
\otimes(\dz{1}\wedge\cdots\wedge\dz{n}\wedge\dzbar{1}\wedge\cdots\wedge \dzbar{n}) .\]
Then, using Eqs.~\eqref{align:basHlocal} and~\eqref{eqt:ldanchorsHonAbd}, one obtains
\[ \representation_{\partialzbar{i}}s^2=-2\tfrac{\partial\thetacoeff{p}{i}}{\partial z^p} s^2, 
\qquad \representation_{\dz{i}}s^2=4\tfrac{\partial\picoeff{i,p}}{\partial z^p}s^2 .\]
The conclusion thus follows immediately.
\end{proof}

\begin{cor}
Locally, the operator $\DDH$ in Eq.~\eqref{Eqt:DDH} is given by
\begin{equation}\label{Eqt:DDHslocal}
\DDH s=(2\tfrac{\partial\picoeff{i,p}}{\partial z^p}\partialz{i} 
-\tfrac{\partial\thetacoeff{p}{i}}{\partial z^p}\dzbar{i})\otimes s
,\end{equation}
where $s$ is defined in Eq.~\eqref{Eqt:s}.
\end{cor}

We are now ready to prove Theorem~\ref{Thm:main1}.
\begin{proof}[Proof of Theorem~\ref{Thm:main1}]
We adopt an inductive approach. First we prove the commutativity 
of Diagram~\eqref{Diagram:stdIso} for $k=0$.

Note that for any $f\in\cinf{\bX,\CC}$, $u\in\sections{\wedge^k\lad}$ 
and $s\in\sections{\module}$, one has
\begin{align*}
\stdiso\bdeesH(f u\otimes s) &= \stdiso\big(f\bdeesH(u\otimes s)
+((\barpd f+\lb{H}{f})\wedge u)\otimes s\big) & \text{by Eq.~\eqref{Eqt:bdeesHuotimess}} \\
&= f\stdiso\bdeesH(u\otimes s)+(\barpd f+\lb{H}{f})\cliffact\stdiso(u\otimes s) . &
\end{align*}

On the other hand, if we write $\anyform=\stdiso(u\otimes s)$, one has
\begin{align*}
& (\barpd+\supperlb{\pd}{\ii{H}})\stdiso(fu\otimes s) & \\
=& (\barpd+\supperlb{\pd}{\ii{H}})(f\anyform) & \\
=& \barpd f\wedge\anyform+f\barpd\anyform
+\ba{H}{f}\cliffact\anyform +f\supperlb{\pd}{\ii{H}}\anyform
& \text{by Eq.~\eqref{Eqt:iHpdfproperty}} \\
=&f(\barpd+\supperlb{\pd}{\ii{H}})\stdiso(u\otimes s)
+(\barpd f+\ba{H}{f})\cliffact\stdiso(u\otimes s). &
\end{align*}

It thus follows that the map 
$\stdiso\circ\bdeesH-(\barpd+\supperlb{\pd}{\ii{H}})\circ\stdiso$
is $\cinf{\bX}$-linear.
Take a local holomorphic chart $(z^1,\dots,z^n)$ 
and write $H$ locally as in Eq.~\eqref{Eqt:Hlocally} 
in Lemma \ref{Lem:localstrAbdsH}. 
Again take $s$ as in Eq.~\eqref{Eqt:s}. 
For $k=0$, we have $\bdeesH s=\DDH s$, which is given locally 
by Eq.~\eqref{Eqt:DDHslocal}. Then, we compute
\begin{align*}
\stdiso(\bdeesH s) =& (2\tfrac{\partial\picoeff{i,p}}{\partial z^p}\partialz{i} -\tfrac{\partial\thetacoeff{p}{i}}{\partial z^p}\dzbar{i})\cliffact(\dz{1}\wedge\cdots\wedge\dz{n}) \\
=& 2\sum_{i=1}^n(-1)^{i+1}\tfrac{\partial\picoeff{i,p}}{\partial z^p} \dz{1}\wedge\cdots\wedge\widehat{\dz{i}}\wedge\cdots\wedge\dz{n}
-\tfrac{\partial\thetacoeff{p}{i}}{\partial z^p}\dzbar{i}\wedge\dz{1}\wedge\cdots\wedge\dz{n}
.\end{align*}
Thus we have
\begin{align*}
(\barpd+\supperlb{\pd}{\ii{H}})s =& \pd\ii{H}(\dz{1}\wedge\cdots\wedge\dz{n}) \\
=& \pd\Big(2\sum_{i<j}(-1)^{i+j-1}\picoeff{i,j}\dz{1}\wedge\cdots\wedge
\widehat{\dz{i}}\wedge\cdots\wedge\widehat{\dz{j}}\wedge\cdots\wedge\dz{n} \\
& \quad +\sum_{p=1}^n(-1)^{p+1}\thetacoeff{p}{i}\dzbar{i}
\wedge\dz{1}\wedge\cdots\wedge\widehat{\dz{p}}\wedge\cdots\wedge\dz{n} \\
& \quad -\omegacoeff{k,l}\dzbar{k}\wedge\dzbar{l}\wedge
\dz{1}\wedge\cdots\cdots\wedge\dz{n}\Big) \\
=& \stdiso(\bdeesH s)
.\end{align*}
It thus follows that Diagram~\eqref{Diagram:stdIso} indeed commutes when $k=0$.

Now assume that we have proved the commutativity of Diagram~\eqref{Diagram:stdIso}
when $k\leq m$ (where $0\leq m\leq 2n-1$). 
To prove the $k=m+1$ case, we consider a section 
$(u\wedge w)\otimes s\in\sections{\wedge^{m+1}\lad\otimes\module}$, 
where $u\in\sections{\lad}$, $w\in\sections{\wedge^m\lad}$ 
and $s\in\sections{\module}$. Then
\begin{align*}
& (\barpd+\supperlb{\pd}{\ii{H}})\stdiso((u\wedge w)\otimes s) & \\
=& (\barpd+\supperlb{\pd}{\ii{H}})(u\cliffact\anyform)
& \text{where }\anyform=w\cliffact s \\
=& \barpd u\cliffact\anyform-u\cliffact\barpd\anyform
+\ba{H}{u}\cliffact\anyform-u\cliffact(\supperlb{\pd}{\ii{H}}\anyform)
& \text{by Eqs.~\eqref{Eqt:barpdulambda} and~\eqref{Eqt:iHpdulambdaproperty}} \\
=& \barpdH u\cliffact\anyform-u\cliffact(\barpd+\supperlb{\pd}{\ii{H}})\anyform \\
=& \stdiso\big((\barpdH u\wedge w)\otimes s\big)-u\cliffact\stdiso\bdeesH(w\otimes s)
& \text{by assumption} \\
=& \stdiso\bdeesH((u\wedge w)\otimes s) . &
\end{align*}
This concludes the proof.
\end{proof}

\subsection{Modular classes}

The modular class of a Lie algebroid was introduced by Evens-Lu-Weinstein~\cite{MR1726784}. 
The following version for complex Lie algebroids appeared
 in the arXiv version \texttt{dgga/9610008}  of~\cite{MR1726784}
but not in the published paper.
 It is also implied in~\cite{math/0703298}. 
The presentation which we give below was communicated to us by Camille Laurent-Gengoux~\cite{LaurentLetter}.

Let $B$ be a complex Lie algebroid over a real manifold $M$,
with $\rk_{\CC}B=r$ and $\dim M=m$. 
Its Evens-Lu-Weinstein module is
$\QB=\wedge^{r}B\otimes\wedge^{m}T^*_\CC M$.

Consider the complex of sheaves
\begin{equation}
\label{Rome} \sxsheaf{0}
\stackrel{\tilde{d}_B}{\longrightarrow} \ssheaf{1}
\stackrel{{d}_B}{\longrightarrow} \ssheaf{2} \cdots
\stackrel{{d}_B}{\longrightarrow} \ssheaf{r},
\end{equation}
where $\sxsheaf{0}$ is the sheaf of nowhere vanishing smooth complex valued functions on $M$; $\ssheaf{\bullet}$ is the sheaf of sections of $\wedge^\bullet B^*$; 
$\dL$ is the usual Lie algebroid cohomology differential; and 
$\dLNonzero f=\dL\log f=\tfrac{\dL f}{f}$, for all $f\in\cinf{U,\CC^{\times}}$, where $U$ is an arbitrary open subset of $M$. 
We denote its hypercohomology by $\tildecohomology{\bullet}(B,\CC)$.
Note that in Eq.~\eqref{Rome}, if we replace $\sxsheaf{0}$ by
$\ssheaf{0}$, the sheaf of smooth complex valued functions on $M$,
and $\tilde{d}_B$ by the usual Lie algebroid differential $d_B$, the
hypercohomology of the resulting complex of sheaves 
\begin{equation}
\label{Athens} \ssheaf{0}
\stackrel{d_B}{\longrightarrow} \ssheaf{1}
\stackrel{d_B}{\longrightarrow} \ssheaf{2} \cdots
\stackrel{d_B}{\longrightarrow} \ssheaf{r},
\end{equation}
is isomorphic to
the usual Lie algebroid cohomology $\cohomology{\bullet}(B,\CC)$
of the complex Lie algebroid $B$ with trivial coefficients $\CC$ since
each $\ssheaf{\bullet}$ is a soft sheaf. 
The exponential sequence 
\[ 0\to\ZZ\to\ssheaf{}\to\sxsheaf{}\to 0 ,\] 
where $\ssheaf{}$ (resp.\ $\sxsheaf{}$) stands for the the complex of sheaves 
\eqref{Athens} (resp.\ \eqref{Rome}) and 
the locally constant sheaf $\ZZ$ is regarded as a complex of sheaves concentrated in degree $0$, 
induces the long exact sequence
\[ \cdots\to\cohomology{i}(M,\ZZ)\to\cohomology{i}(B,\CC) 
\to\tildecohomology{i}(B,\CC)\to\cohomology{i+1}(M,\ZZ)\to\cdots \]

Note that $\tildecohomology{\bullet}(B,\CC)$ can be computed as
the total cohomology of the \v{C}ech double complex
\begin{equation} \label{Diagram:CechLsNonzero} 
\xymatrix{ 
\cdots & \cdots & \cdots & \\ 
\sxchech{2}{0} \ar[r]^{\dLNonzero} \ar[u]^{\cechdelta} &
\schech{2}{1} \ar[r]^{\dL} \ar[u]^{\cechdelta} &
\schech{2}{2} \ar[r]^{\dL} \ar[u]^{\cechdelta} & \cdots \\ 
\sxchech{1}{0} \ar[r]^{\dLNonzero} \ar[u]^{\cechdelta} &
\schech{1}{1} \ar[r]^{\dL} \ar[u]^{\cechdelta} &
\schech{1}{2} \ar[r]^{\dL} \ar[u]^{\cechdelta} & \cdots \\ 
\sxchech{0}{0} \ar[r]^{\dLNonzero} \ar[u]^{\cechdelta} &
\schech{0}{1} \ar[r]^{\dL} \ar[u]^{\cechdelta} &
\schech{0}{2} \ar[r]^{\dL} \ar[u]^{\cechdelta} & \cdots 
} \end{equation}
where $\cover=\set{U_{i}}_{i\in I}$ is a good open cover of $M$ and
$\cechdelta$ is the usual \v{C}ech coboundary operator.

Let $(U_i)_{i\in I}$ be a good open cover of $M$, 
and $\omega_i$ a nowhere vanishing section of $\QB$ over $U_i$. 
For all $i,j\in I$, there exists a unique nowhere vanishing function
$f_{ij}\in\cinf{U_{ij},\CC^\times}$ such that $\omega_i=f_{ij}\omega_j$. 
It is clear from the construction that
\[ f_{ij} f_{jk} f_{ki}=1 .\]
Let $\xi_i\in\sections{B^*|_{U_i}}$ be the modular $1$-form on ${U_i}$ corresponding to $\omega_i$. 
That is, we have $\nabla_X\omega_i=\duality{\xi_i}{X}\omega_i$ for all $X\in\sections{B|_{U_i}}$, 
where $\nabla$ denotes the canonical representation of $B$ on $Q_B$ of~\cite{MR1726784}. 
It thus follows that \[ \xi_i=\xi_j+\tfrac{d_B f_{ij}}{f_{ij}}=\xi_j+\tilde{d}_B f_{ij} .\]
As a consequence, $(\xi_i,f_{ij})$ is a 1-cocycle 
of the double complex~\eqref{Diagram:CechLsNonzero}, 
and therefore defines a class in $\tildecohomology{1}(B,\CC)$.

\begin{defn}
The class in $\tildecohomology{1}(B,\CC)$ defined by $[(\xi_i,f_{ij})]$ is called the \emph{modular class} of the complex Lie algebroid $B$, and denoted $\modular(B)$.
\end{defn}

\begin{lem}\label{Tokyo}
Consider the long exact sequence
\[ \cdots \to \cohomology{1}(B,\CC) \to \tildecohomology{1}(B,\CC)
\xto{\tau} \cohomology{2}(M,\ZZ) \to \cdots \] 
The image of the modular class $\modular(B)$ under $\tau$ is the first Chern class $c_1(\QB)$ of $\QB$. When $c_1(\QB)=0$, the modular class $\modular(B)$ is the image of a class in
$\cohomology{1}(B,\CC)$, which is defined exactly in the same way
using a global nowhere vanishing section, as the usual modular class
in~\cite{MR1726784}.
\end{lem}

A complex Lie algebroid $B$ is said to be \textbf{unimodular} if its modular class vanishes.
The following result follows immediately from Lemma~\ref{Tokyo}.

\begin{cor}
A complex Lie algebroid $B$ is unimodular if and only if $c_1(\QB)=0$ 
and for any fixed nowhere vanishing section $\omega\in\sections{\QB}$, 
the modular section $\xi\in\sections{B^*}$ defined by,
\[ \nabla_X\omega=\duality{\xi}{X}\omega \qquad(\forall X\in\sections{B}) \]
is a coboundary, i.e. $\xi=\dL f$ for some $f\in\cinf{M,\CC}$.
\end{cor}

As a consequence, a complex Lie algebroid $B$ is unimodular if and
only if $\QB$ is isomorphic to the trivial module $\CC$.

\begin{prop}
When $B=T^{0,1}\bX\bowtie\lad^{1,0}$ is the derived complex Lie
algebroid~\cites{MR2439547,arXiv:0903.5065} of a holomorphic Lie algebroid
$\lad$ over $\bX$, $B$ is a unimodular complex Lie algebroid 
if and only if $\lad$ is 
a unimodular holomorphic Lie algebroid, i.e. $Q_\lad$ is trivial as a holomorphic
line bundle and there exists a holomorphic global section $\omega$
of $Q_\lad$ such that $\nabla_X\omega=0$ for all $X\in\lad$.
\end{prop}

\begin{defn}
An extended Poisson manifold $(\bX,H)$ is unimodular if its
corresponding complex Lie algebroid $\dal_H$ is unimodular.
\end{defn}

According to Theorem~\ref{Thm:main1}, we have

\begin{prop}
An extended Poisson manifold $(\bX,H)$ is unimodular 
if and only if there
exists a nowhere vanishing $(n,0)$-form $\omega\in\forms{n,0}$
such that \[ \barpd\omega+\supperlb{\pd}{\ii{H}}\omega=\barpd\omega+\pd\ii{H}\omega=0 .\]
\end{prop}

\begin{rmk}
It is clear that, when $H=0$, $(\bX,H)$ is unimodular means
that $\bX$ is Calabi-Yau. Thus one can consider a unimodular
extended Poisson manifold $(\bX,H)$ as a generalized Calabi-Yau manifold.
\end{rmk}

As an immediate consequence of the discussion above, we have
\begin{cor}
For any unimodular extended Poisson manifold $(\bX,H)$ of complex dimension $n$, 
we have \[ \homology{k}(\bX,H)\cong\cohomology{2n-k}(\bX,H) .\]
\end{cor}

\end{document}